\documentclass[12pt]{article}
\usepackage{amsmath,amsthm}
\usepackage{amssymb,bbm}
\usepackage[pdftex]{graphicx}
\usepackage[T1]{fontenc}

\newtheorem{theorem}{Theorem}[section]
\newtheorem{corollary}[theorem]{Corollary}

\newtheorem{proposition}[theorem]{Proposition}
\newtheorem{problem}[theorem]{Problem}

\theoremstyle{definition}

\newtheorem*{xrem}{Remark}

\numberwithin{equation}{section}

\newcommand{\Set}[2]{
\left\{#1\,\middle|\,#2\right\}}

\newcommand{\one}{\mathbbm{1}}
\newcommand{\two}{\mathbbm{2}}
\newcommand{\R}{\mathbb{R}}
\newcommand{\M}{\mathcal{M}}
\newcommand{\B}{\mathcal{B}}
\newcommand{\X}{\mathcal{X}}
\newcommand{\W}{\mathcal{W}}
\newcommand{\dGH}{d_{\mathrm{GH}}}
\newcommand{\dis}{\mathop{\mathrm{dis}}\nolimits}
\newcommand{\diam}{\mathop{\mathrm{diam}}\nolimits}
\newcommand{\GHC}{\mathop{\mathrm{GHC}}\nolimits}
\newcommand{\Sum}{\mathop{\mathfrak{S}}\nolimits}

\begin{document}

\title{A sequence of compact metric spaces and an isometric embedding into the Gromov-Hausdorff space}

\author{Takuma Byakuno
}
\date{}
\maketitle

\renewcommand{\thefootnote}{}
\thefootnote{2020 \emph{Mathematics Subject Classification}: Primary 30L05; Secondary 54E35, 53C23.}
\thefootnote{\emph{Key words and phrases}: Gromov-Hausdorff space, Isometric embeddings.}
\renewcommand{\thefootnote}{\arabic{footnote}}
\setcounter{footnote}{0}

\begin{abstract}
For a convergent series with positive terms, we prove that the $\ell^\infty$ product space of bounded subspaces of the Gromov-Hausdorff space can be isometrically embedded into the Gromov-Hausdorff space, where each subspace consists of compact metric spaces with the diameter less than or equal to the term.
\end{abstract}

\section{Introduction}
The {\it Gromov-Hausdorff space} is a metric space consisting of a certain set $\M$ and a distance function $\dGH$. 
All compact metric spaces belong essentially to the set $\M$, and $\dGH(X, Y)$ is determined by the difference between the shapes of two compact metric spaces $X, Y$.\\

The distance function $\dGH$ is called the {\it Gromov-Hausdorff distance}, and we can easily calculate the Gromov-Hausdorff distance in several cases.
We can define a compact metric space $tX\in\M$ as a scale transformation of $X\in\M$ by $t>0$:
Let $d_X$ be a distance function on $X$. 
Then $tX$ and its distance function $d_{tX}$ are defined by
$tX=X$, and $d_{tX}(x,y)=t d_X(x,y)$ respectively. 
It is straightforward to see 
\begin{itemize}
\item[(1)]
$\dGH(X,\one)=2^{-1}\diam X$ for $X\in\M$,
\item[(2)]
$\dGH(tX,tY)=t\dGH(X,Y)$ for $X,Y\in\M$ and $t>0$,
\item[(3)]
$\dGH(tX,sX)=2^{-1}|t-s|\diam(X)$ for $X\in\M$ and $s,t>0$,
\end{itemize}
where $\one$ is a one-point metric space. \\

Now, we can consider the diameter as a function $\diam:\M\to\R$, and we can show this function is a 2-lipschitz function by (1):
\[
|\diam(X)-\diam(Y)|\leq 2\dGH(X,Y)\,.
\]
We can also consider the Gromov-Hausdorff space as a cone with the vertex at $\one$,
and the diameter function $\diam:\M\to\R$ expresses the depth of this cone:
If we define the subset  $\B_s$, $\B_{\leq s}$ by
\[
\B_s=\Set{X\in\M}{\diam(X)=s}
\ ,\quad
\B_{\leq s}=\Set{X\in\M}{\diam(X)\leq s}\,,
\]
then $\B_{\leq s}$ is a subcone over $\B_s$, See Figure \ref{figure1}.
The isometry group of the Gromov-Hausdorff space is trivial by \cite{Tuzhilin2}.\\
\begin{figure}[h]
\centering
\includegraphics[width=0.6\textwidth]{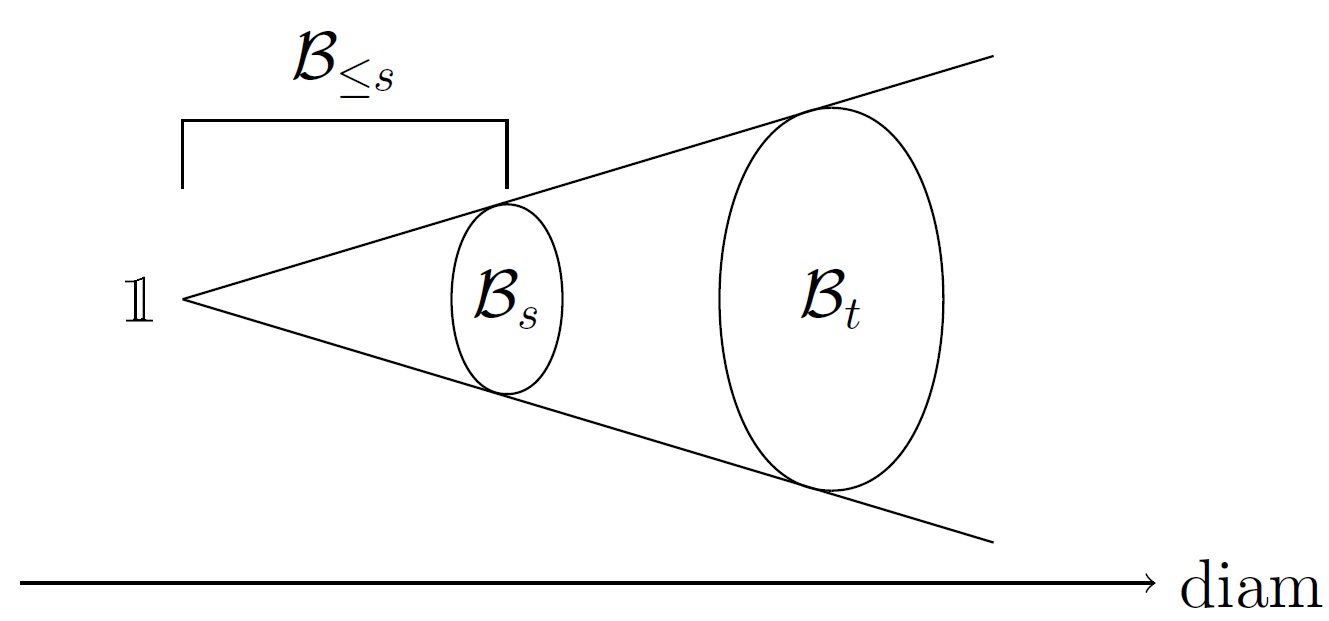}
\caption{The Gromov-Hausdorff space.}
\label{figure1}
\end{figure}

It is known that the Gromov-Hausdorff space $\M$ has a dense countable set, and $\M$ is therefore separable. 
In addition, we can show that $\M$ is complete, but $\M$ is not isometric to the Urysohn universal space. 
See \cite[Section 3]{Tuzhilin1} and \cite{Tuzhilin2}.
Recently, Nakajima, Yamauchi, and Zava studied the topological dimension of $\M$ in \cite{Nakajima} and calculated the topological dimension of $\M_{\leq n}$, which is a subspace consisting of compact metric spaces with at most $n$ points.
\\

In 2017, Iliadis, Ivanov, and Tuzhilin \cite{Tuzhilin1} showed that all finite metric spaces can be isometrically embedded into the Gromov-Hausdorff space, which means this space is universal for a class of all finite metric spaces.\\

According to \cite{Tuzhilin1}, there are many open problems related to the geometrical properties of $\M$. 
Among them, in this paper, we focus on
\begin{problem}
Can we isometrically embed all compact metric spaces into the Gromov-Hausdorff space $\M$? 
\end{problem}
Iliadis, Ivanov, and Tuzhilin's result can be considered as a partial answer to this problem, and this partial answer can be generalized to the following as an easy application of \cite[Theorem 4.1]{Tuzhilin1}.
\begin{proposition}
\label{EBs2GHs}
All bounded metric subspaces of an $\ell^\infty$ normed space $\R^n$ can be isometrically embedded into the Gromov-Hausdorff space. 
\end{proposition}

In this paper, we will prove the following theorem.
\begin{theorem}
\label{theorem1}
Let $c=\sum_{n=1}^\infty r_n$ be a convergent series with $r_n>0$.
Then the direct product set $\B_{\leq r}=\prod_{n=1}^\infty\B_{\leq r_n}$ equipped with $\ell^\infty$ (or supremum) distance can be embedded into $\B_{36c-3r_1}$.
Namely, we can construct a map $\Sum_r:\B_{\leq r}\to\B_{36c-3r_1}$ satisfying
\[
\dGH(\Sum_r(X),\Sum_r(Y))=\max_{n=1,2,\ldots}\dGH(X_n,Y_n)
\]
for all $X=(X_n)_{n=1}^\infty, Y=(Y_n)_{n=1}^\infty\in\B_{\leq r}$.
\end{theorem}

Ishiki showed that the Hilbert cube can be topologically embedded into $\M$ in \cite{Ishiki}.
We can show this proposition by Theorem \ref{theorem1} since $\B_{\leq r}$ includes an isometric copie of $\prod_{n=1}^\infty[0,r_n]$ endowed with $\ell^\infty$ distance. 
Theorem \ref{theorem1} also implies
\begin{corollary}
\label{theorem2}
Let $c=\sum_{n=1}^\infty r_n$ be a convergent series with $r_n>0$.
Then a subspace $\Set{(x_n)_{n=1}^\infty}{0\leq x_n\leq r_n}$ of the space of bounded sequences $\ell^\infty$ can be isometrically embedded into the Gromov-Hausdorff space.
\end{corollary}

\section{Preliminaries}
Let $X,Y$ be sets, and $R$ a subset of $X\times Y$.
For $A\subset X$, we define $R[A]$ as an {\it image} of $A$:
\[
R[A]=\Set{y\in Y}{\mathrm{there\ exists}\ a\in A\ \mathrm{such\ that}\ (a,y)\in R },
\]
and we define $R^{-1}$ as an {\it inverse} of $R$:
\[
R^{-1}=\Set{(y,x)}{(x,y)\in R}\subset Y\times X.
\]
Similarly, an {\it inverse image} $R^{-1}[B]$ is also defined.\\

Using this notation, we can define $\GHC(X,Y)$ by
a set of all $R\subset X\times Y$ satisfying
$R[\{x\}]\neq\emptyset$ and $R^{-1}[\{y\}]\neq\emptyset$ for all $x\in X$ and $y\in Y$.
For $R\in\GHC(X,Y)$, we have $R^{-1}\in\GHC(X,Y)$.
The {\it distortion} of $R\in\GHC(X,Y)$ is a real or infinite value 
\[
\dis(R)=\dis(X,Y;R)=\sup\Set{|d_X(x,\xi)-d_Y(y,\eta)|}{(x,y),(\xi,\eta)\in R},
\]
where $d_X,d_Y$ are the distance functions on $X,Y$, respectively.\\

The Gromov-Hausdorff distance is typically defined by using the ``Hausdorff distance'', but we can alternatively define the {\it Gromov-Hausdorff distance} between two metric spaces $X, Y$ by
\[
\dGH(X,Y)=\frac12\inf_{R\in\GHC(X,Y)}\dis(R)\,.
\]
We write the {\it diameter} of a metric space $(X,d)$ as $\diam(X)$:
\[
\diam(X)=\diam(X,d)=\sup_{x_0,x_1\in X}d(x_0,x_1).
\]
Then the Gromov-Hausdorff distance satisfies several distance axioms:
\begin{itemize}
\item[(1)]
$0\leq\dGH(X,Y)\leq\max\{\diam(X),\diam(Y)\}$ for all metric spaces $X,Y$,
\item[(2)]
$\dGH(X,Y)=\dGH(Y,X)$  for all metric spaces $X,Y$,
\item[(3)]
$\dGH(X,Y)\leq\dGH(X,Z)+\dGH(Z,Y)$  for all metric spaces $X,Y,Z$,
\item[(4)]
For all compact metric spaces $X,Y$, 
$\dGH(X,Y)=0$ if and only if $X$ is isometric to $Y$.
\end{itemize}
Therefore, $\dGH$ is a distance function on a set $\M$ satisfying 
\begin{itemize}
\item[(M1)]
For every $X\in\M$, $X$ is a (nonempty) compact metric space,
\item[(M2)]
For an arbitrary compact metric space $X$,
we can find a suitable $Y\in\M$ such that $Y$ is isometric to $X$,
\item[(M3)]
For every $X,Y\in\M$, we obtain $X=Y$ if $X$ is isometric to $Y$.
\end{itemize}
There exists such a set $\M$ since an arbitrary compact metric space can be isometrically embedded into the space of bounded sequences $\ell^\infty$.
The metric space $(\M,\dGH)$ is called the {\it Gromov-Hausdorff space}.
By the definition of $\M$, for a compact metric space $X$, there exists a unique $Y\in\M$ such that $X$ is isometric to $Y$.
Then we may consider $[X]=Y$.\\

For simplicity, we write $\omega$ for the ordinal number of the well-ordered set $\{0,1,\ldots\}$ (the set of all non-negative integers):
\[
0\leq1\leq2\leq\cdots\leq n\leq\cdots\leq\omega\leq\omega+1\leq\cdots\,.
\]
Let $\{r_\lambda\}_{\lambda\in\Lambda}$ ($\Lambda$ is a set) be a family of points such that $r_\lambda$ is a non-negative real number for each $\lambda\in\Lambda$.
Then we can also define a series of
 $\{r_\lambda\}_{\lambda\in\Lambda}$ by
\[
\sum_{\lambda\in\Lambda}r_\lambda
=\sup\Set{\sum_{\lambda\in P}r_\lambda}{P\ \mathrm{is\ a\ finite\ subset\ of}\ \Lambda}\,.
\]
This paper uses this notation only if $\Lambda$ consists of ordinal numbers.
For instance, if the sequence $(r_n)_{n=1}^\infty$ has the convergent series, we have
\[
\sum_{0\leq\lambda<\omega+1}=r_0+\left(\sum_{n=1}^\infty r_n\right)+r_{\omega}\,.
\]

\section{Proof of Theorem \ref{theorem1}}
Firstly, we shall prove Theorem \ref{theorem1} and remark on an isometric embedding of a subspace consisting of a compact metric space with constant Hausdorff dimension.
Finally, we shall prove Corollary \ref{theorem2}. 
\begin{proof}[Proof of Theorem \ref{theorem1}]
Let $c=\sum_{n=1}^\infty r_n$ be a convergent series with $r_n>0$.
Then we set $r'_n=r_n+r_{n+1}$, and
we define $\Sum_r(X)\in\M$ for $X=(X_n)_{n=1}^\infty\in\B_{\leq r}=\prod_{n=1}^\infty \B_{\leq r_n}$: Firstly, we write $X_0,X_\omega,X_{\omega+1}$ for one-point metric spaces, where $X_0,X_\omega,X_{\omega+1}$ have one points $p_0,p_\omega,p_{\omega+1}$ and distance functions $d_0,d_\omega,d_{\omega+1}$, respectively.
We also write a distance function on $X_n$ as $d_n$.
Then we put $\Sum_r(X)=\bigsqcup_{\lambda\in\Lambda}X_\lambda$, and take a unique distance function $d_X$ on $\Sum_r(X)$ such that $d_X$ is an extension of $d_\lambda$ for all $0\leq\lambda\leq\omega+1$ and satisfies
\[
d_X(x_\alpha,x_\beta)=3\sum_{\alpha\leq \lambda<\beta}r'_\lambda
\]
for $\alpha,\beta\in\Lambda$ with $\alpha<\beta$, $x_\alpha\in X_\alpha$ and $x_\beta\in X_\beta$, where $r'_0=2c$ and $r'_\omega=8c$.
This metric space $(\Sum_r(X),d_X)$ looks like Figure \ref{figure2}.
\begin{figure}[h]
\centering
\includegraphics[width=0.9\textwidth]{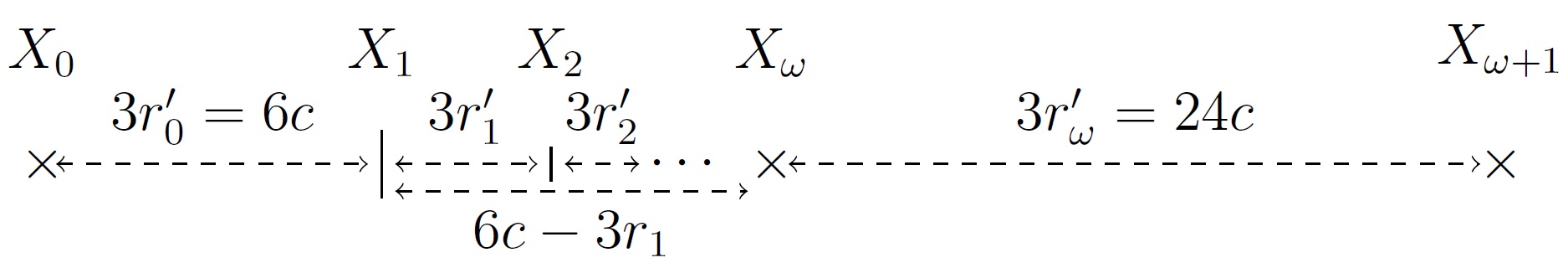}
\caption{The definition of $\Sum_r(X)$.}
\label{figure2}
\end{figure}
\\
The metric space is compact, thus we can consider $\Sum_r(X)=[\Sum_r(X)]\in\M$.
We also have $\Sum_r(X)\in\B_{36c-3r_1}$:
\[
\diam(\Sum_r(X))=d_X(p_0,p_{\omega+1})=6c+6c-3r_1+24c=36c-3r_1\,.
\]

We shall prove that $\Sum_r:\B_{\leq r}\to\B_{36c-3r_1}$ is an isometric embedding, that is, $\dGH(\Sum_r(X),\Sum_r(Y))=\sup_{n=1,2,\ldots}\dGH(X_n,Y_n)$ for $X=(X_n)_{n=1}^\infty,Y=(Y_n)_{n=1}^\infty\in\B_{\leq r}$.
We put $S=\sup_{n=1,2,\ldots}\dGH(X_n,Y_n)$, and note that $X_0=Y_0=\{p_0\}$, $X_\omega=Y_\omega=\{p_0\}$ and $X_{\omega+1}=Y_{\omega+1}=\{p_0\}$.\\

For $R_\lambda\in\GHC(X_\lambda,Y_\lambda)$ ($0\leq\lambda\leq\omega+1$), we have $R=\bigcup_{0\leq\lambda\leq\omega+1}R_\lambda\in\GHC(\Sum_r(X),\Sum_r(Y))$ and $\dis(R)\leq\sup_{\lambda\in\Lambda}\dis(R_\lambda)
=\sup_{n=1,2,\ldots}\dis(R_n)$,
since $\dis(R_0)=\dis(R_\omega)=\dis(R_{\omega+1})=0$. 
Thus $\dGH(\Sum_r(X),\Sum_r(Y))\leq S$.\\

Now, we may assume $S\neq0$.
Since $\max\{\diam(X_n),\diam(Y_n)\}\leq r_n$ and $\lim_{n\to\infty}r_n=0$, There exists a number $N$ such that $\dGH(X_n,Y_n)<S$ for $n\geq N$.
Therefore, we can take a number $K=1,2,\ldots$ such that $S=\dGH(X_K,Y_K)$.\\

Suppose, for the sake of contradiction, that $\dGH(\Sum_r(X),\Sum_r(Y))<S$.
Then there exists $R\in\GHC(\Sum_r(X),\Sum_r(Y))$ such that $\dis(R)<2S$.\\ 

For $y_0\in R[X_0]$ and $y_{\omega+1}\in R[X_{\omega+1}]$, we shall prove $y_0\in Y_0$ and $y_{\omega+1}\in Y_{\omega+1}$.
Since $\dis(R)<2S\leq 2c$,
\[
d_Y(y_0,y_{\omega+1})
\geq d_X(p_0,p_{\omega+1})-\dis(R)
>(36c-3r_1)-2c
=34c-3r_1.
\]
If there exists an ordinal number $\lambda$ such that  $y_0,y_{\omega+1}\in Y_\lambda$, we have
\[
34c-3r_1
<d_Y(y_0,y_{\omega+1})
=d_\lambda(y_0,y_{\omega+1})
\leq\diam(Y_\lambda)\leq c\,,
\]
which is a contradiction. Thus, there are ordinal numbers $\alpha,\beta$ such that $y_0\in Y_\alpha$ and $y_{\omega+1}\in Y_\beta$, but $\alpha\neq\beta$. If $\max\{\alpha,\beta\}<\omega+1$, then we have
\[
34c-3r_1
<d_Y(y_0,y_{\omega+1})
=3\sum_{\min\{\alpha,\beta\}\leq\lambda<\max\{\alpha,\beta\}}r'_\lambda
\leq 3\sum_{0\leq\lambda<\omega}r'_\lambda
=12c-3r_1\,,
\]
which is a contradiction to $c>0$. Therefore $\max\{\alpha,\beta\}=\omega+1$.
In addition, if $0<\min\{\alpha,\beta\}$, then
\[
34c-3r_1
<d_Y(y_0,y_{\omega+1})
=3\sum_{\min\{\alpha,\beta\}\leq\lambda<\omega+1}r'_\lambda
\leq 3\sum_{0<\lambda\leq\omega+1}r'_\lambda
=30c-3r_1\,,
\]
which is also a contradiction to $c>0$. As a consequence, we have $\max\{\alpha,\beta\}=\omega+1$ and $\min\{\alpha,\beta\}=0$.\\

Suppose $\alpha=\omega+1$ and $\beta=0$, that is, $y_0\in Y_{\omega+1}$ and $y_{\omega+1}\in Y_0$. 
Taking $y_\omega\in R[X_\omega]$, we have
\[
d_Y(y_\omega,y_{\omega+1})\geq d_X(p_\omega,p_{\omega+1})-\dis(R)>24c-2c=22c\,.
\]
If $y_\omega\in Y_0$, we have $c<d_Y(y_\omega,y_{\omega+1})\leq\diam(Y_0)=0$, which is a contradiction.
In addition, if $y_0\in Y_\delta$ ($0<\delta<\omega+1$), we can also obtain a contradiction:
\[
22c
<d_Y(y_\omega,y_{\omega+1})
=3\sum_{0\leq\lambda<\delta}r'_\lambda
<3\sum_{0\leq\lambda<\omega}r'_\lambda
=12c-3r_1\,.
\]
Therefore $\alpha=0$ and $\beta=\omega+1$ since $\min\{\alpha,\beta\}=0$ and $\max\{\alpha,\beta\}=\omega+1$.
It means $y_0\in Y_0$ and $y_{\omega+1}\in Y_{\omega+1}$.\\

For $x_K\in X_K$ and $y_K\in R[\{x_K\}]$, we shall prove $y_K\in Y_K$.
There exists an ordinal number $\delta$ such that $y_K\in Y_\delta$.
Since $y_0\in Y_0$ and 
\[
d_Y(y_0,y_K)\geq d_X(p_0,x_K)-\diam(R)>3\sum_{0\leq\lambda<K}r'_\lambda-2c>6c-2c=4c>0\,,
\]
we have $\delta>0$. Therefore, if $\delta\neq K$, we have $3r_K\leq\dis(R)<2S$: 
\begin{align*}
\dis(R)
\geq&|d_Y(y_0,y_K)-d_X(p_0,x_K)|
=
\left|3\sum_{0\leq\lambda<\delta}r'_\lambda-3\sum_{0\leq\lambda<K}r'_\lambda\right|\\
=&\begin{cases}
\displaystyle
3\sum_{K\leq\lambda<\delta}r'_\lambda
&\mathrm{if:}K<\delta\\
\displaystyle
3\sum_{\delta\leq\lambda<K}r'_\lambda&\mathrm{if:}\delta<K
\end{cases}
\geq\begin{cases}
\displaystyle
3r'_K
&\mathrm{if:}K<\delta\\
\displaystyle
3r'_{K-1}&\mathrm{if:}\delta<K
\end{cases}
\geq3r_K\,.
\end{align*}
However, Since we put $S=\dGH(X_K,Y_K)$, 
\[
2S=2\dGH(X_K,Y_K)\leq2\max\{\diam(X_K),\diam(Y_K)\}\leq2r_K<3r_K<2S\,.
\]
Thus, we have $\delta=K$, that is, $y_K\in Y_K$.\\

Consequently, we have $R[X_K]\subset Y_K$.
In addition, we can obtain $R^{-1}[Y_K]\subset X_K$ by $\dis(R^{-1})=\dis(R)$ similarly.
Therefore we can define $R_K\in\GHC(X_K,Y_K)$ by $R_K=R\cap (X_K\times Y_K)\in\GHC(X_K,Y_K)$.
Since $R_K\subset R$,
\[
S=\dGH(X_K,Y_K)\leq\frac12\dis(R_K)\leq\frac12\dis(R)<S\,,
\]
which is a contradiction. Hence $\dGH(\Sum_r(X),\Sum_r(Y))=S$.
\end{proof}
\begin{xrem}
The metric space $\Sum_r(X)$ is the disjoint union of $X_1,X_2,\ldots.$
Let $\X$ be a subset of $\M$ such that $\bigsqcup_{0\leq\lambda\leq\omega+1} X_\lambda\in\M$ for $X_1,X_2,\ldots\in\X$ and $X_0=X_\omega=X_{\omega+1}=\one$.
Then, putting $\X_r=\Set{X\in \X}{\diam(X)\leq r}$, we can isometrically embed $\prod_{n=1}^\infty \X_{r_n}$ into $\X$.
For example, 
\[
\W_a=\Set{X\in\M}{\mathrm{the\ Hausdorff\ dimension\ of}\ X\ \mathrm{is}\ a}\,,\quad a\geq0\,.
\]
Now, we define $\dim_H(X)$ as the Hausdorff dimension of $X$. 
By the countable stability of the Hausdorff dimension, we have
\[
\dim_H(\Sum_r(X))
=\sup_{0\leq\lambda\leq\omega+1}\dim_H(X_\lambda)
=\sup_{n=1,2,\ldots}\dim_H(X_n)\,,
\]
that is $\Sum_r(X)\in\W_a$ for $X=(X_n)_{n=1}^\infty\in\prod_{n=1}^\infty(\W_a)_{r_n}$.
Therefore, $\prod_{n=1}^\infty(\W_a)_{r_n}$ can be isometrically embedded into $\W_a$.
\end{xrem}
\begin{proof}[Proof of Corollary \ref{theorem2}]
Let $\two=\{-1,1\}$ be a subspace of the Euclidean space $\R$. 
Then, for $s,t>0$,
\[
\dGH(s\two,t\two)=|s-t|\frac12\diam(\two)=|s-t|\frac{|1-(-1)|}2=|s-t|.
\]
Hence, we can isometrically embed $\Set{(x_n)_{n=1}^\infty}{0\leq x_n\leq r_n}$ into $\prod_{n=1}^\infty\B_{\leq r_n}$, and into $\B_{36c-3r_1}\subset\M$ by Theorem \ref{theorem1}.
\end{proof}
\section{Acknowledgments}
\*\indent This research was financially supported by the Kansai University Grant-in-Aid for research progress in graduate course, 2025.
The author acknowledges this support.
Dr.Yoshito Ishiki, Professor Toshihiro Shoda, and Supervisor Atsushi Fujioka reviewed his paper and encouraged him.
The author thanks them for their kindness.

Takuma Byakuno: the Graduate School of Science and Engineering, Kansai University, 3-3-35 Suita, 564-8680 Osaka, Japan, E-mail: k241344@kansai-u.ac.jp

\end{document}